\documentclass[12pt]{amsart}
\usepackage{amsmath, amsthm, amssymb}
\usepackage[colorlinks=true]{hyperref}

\theoremstyle{plain} 
\newtheorem{theorem}{Theorem}[section]
\newtheorem{corollary}[theorem]{Corollary}
\newtheorem{lemma}[theorem]{Lemma}
\newtheorem{proposition}[theorem]{Proposition}

\theoremstyle{definition}
\newtheorem{defn}[theorem]{Definition}
\newtheorem{example}[theorem]{Example}

\theoremstyle{remark}
\newtheorem{rem}[theorem]{Remark}

\numberwithin{equation}{section}

\DeclareMathOperator{\Hom}{Hom}

\DeclareMathOperator{\soc}{soc}
\DeclareMathOperator{\ann}{ann}

\newcommand{\ra}{\rightarrow}
\newcommand{\field}[1]{\mathbb{#1}}
\newcommand{\charac}[1]{\widehat{#1}}
\newcommand{\addchar}[1]{\charac{#1}}
\newcommand{\F}{\ensuremath{\field{F}}}
\newcommand{\C}{\ensuremath{\field{C}}}

\newcommand{\hr}{\ensuremath{\addchar{R}}}
\newcommand{\size}[1]{\lvert #1 \rvert}

\newcommand{\al}{\alpha}
\newcommand{\be}{\beta}
\newcommand{\ga}{\gamma}

\newcommand{\U}{\ensuremath{\mathcal{U}}}

\newcommand{\A}{\ensuremath{\mathcal{A}}}

\newcommand{\ah}{\ensuremath{\charac{A}}}

\newcommand{\RR}{\ensuremath{\mathcal{R}}}

\DeclareMathOperator{\lt}{lt}
\DeclareMathOperator{\rt}{rt}

\newcommand{\grt}{G_{\rt}}

\DeclareMathOperator{\GL}{GL}

\newcommand{\ca}{\circledast}

\begin{document}
\title[Extension Theorem]{The Extension Theorem for Bi-invariant Weights over Frobenius Rings and Frobenius Bimodules}

\author[O.~W.~Gnilke, et al.]{Oliver W. Gnilke}
\address{Department of Mathematics and Systems Analysis, Aalto University, P.O. Box 11100, FI-00076 Aalto, Finland}
\email{oliver.gnilke@aalto.fi}

\author[]{Marcus Greferath}
\email{marcus.greferath@aalto.fi}

\author[]{Thomas Honold}
\address{ZJU-UIUC Institute, Zhejiang University
718 East Haizhou Road, Haining, China}
\email{honold@zju.edu.cn}

\author[]{Jay A. Wood}
\address{Department of Mathematics, Western Michigan University, 1903 W. Michigan Ave., Kalamazoo, MI 49008-5248, USA}
\email{jay.wood@wmich.edu}

\author[]{Jens Zumbr\"agel}
\address{Faculty of Computer Science and Mathematics, University of Passau, Innstra{\ss}e 33, 94032 Passau, Germany} 
\email{jens.zumbraegel@uni-passau.de}

\dedicatory{In memory of our colleague and friend Aleksandr A. Nechaev, \oldstylenums{1945}--\oldstylenums{2014}}


\begin{abstract}
We give a sufficient condition for a bi-invariant weight on a Frobenius bimodule to satisfy the extension property.  This condition applies to bi-invariant weights on a finite Frobenius ring as a special case.  The complex-valued functions on a Frobenius bimodule are viewed as a module over the semigroup ring of the multiplicative semigroup of the coefficient ring.
\end{abstract}

\keywords{Frobenius ring, Frobenius bimodule, linear code, extension theorem, semigroup ring, M\"obius function}


\maketitle

\section{Introduction}
When coding theory was first developed in the 1940s and 50s, linear codes were defined over finite fields using the Hamming weight.  Coding theory has evolved since those early years, and linear codes are now often defined over alphabets that are finite modules over a finite ring using a more general weight on the module. 

A fundamental question about linear codes, regardless of the level of abstraction, is: \emph{When should two linear codes be considered equivalent?}  There are two competing notions of equivalence.  One notion is that two linear codes are equivalent if there exists a weight-preserving monomial transformation of the ambient space that takes one code to the other.  The other notion is that two linear codes are equivalent if there exists a weight-preserving linear isomorphism between the two codes.  If two linear codes are equivalent in the first sense, then they are equivalent in the second sense.  Indeed, the restriction of the weight-preserving monomial transformation to the codes defines a weight-preserving linear isomorphism between them.

The converse, i.e., if two linear codes are equivalent in the second sense, then they are equivalent in the first sense, is called the \emph{extension problem}, because the question is whether every weight-preserving linear isomorphism between two linear codes extends to a weight-preserving monomial transformation of the ambient space.  When the extension problem for a given alphabet and a given weight has a positive solution, i.e., when every weight-preserving linear isomorphism extends to a weight-preserving monomial transformation, we often say that the \emph{extension theorem holds} for that alphabet and weight.

The extension theorem is known to hold in a number of situations.  The earliest version, over finite fields with the Hamming weight, is due to MacWilliams \cite{macwilliams:theorem, macwilliams:thesis}.  The extension theorem holds when the alphabet is a finite Frobenius ring using the Hamming weight \cite{wood:duality} or the homogeneous weight \cite{greferath-schmidt:combinatorics}, and the extension theorem holds over Frobenius bimodules over any finite ring using either the Hamming or the homogeneous weight \cite{greferath-nechaev-wisbauer:modules}.

There has been considerable work on determining tractable conditions on a general weight in order that the extension theorem hold for that weight.  Some of the earliest work involved a nonsingularity condition on a matrix determined by both the alphabet and the weight that allowed one to reduce the problem to the known extension properties of a so-called symmetrized weight composition \cite{wood:waterloo}.

This paper, drawing on ideas introduced in \cite{GH-OnWeights}, \cite{GH-MonomialExtensions} and \cite{MR3016561}, describes conditions on a weight with maximal symmetry (called a \emph{bi-invariant} weight) over a Frobenius bimodule so that the extension problem reduces to the known extension properties of the homogeneous weight.  The conditions found generalize those for weights defined over a finite principal ideal ring in \cite[Theorem~4.4]{MR3207471}.  The reduction process is accomplished by viewing any weight on a Frobenius bimodule as an element of a complex vector space which is itself a module over the complex semigroup ring of the multiplicative semigroup of the coefficient ring of the Frobenius bimodule.  The additional algebraic structure of modules over semigroup rings allows us to compare the extension properties of weights that are scalar multiples over the semigroup ring.  The key observation, dating from \cite{GH-OnWeights}, is that if $w'$ is a scalar multiple of $w$, both having maximal symmetry, and if the extension theorem holds for $w'$, then the extension theorem holds for $w$.  We apply this idea with $w'$ equaling the homogeneous weight, and the conditions we determine are those that allow us to find a scalar $\ga$ such that $w \ga = w'$.

Here is a short guide to the rest of the paper.  Semigroup rings and certain modules are introduced in Section~\ref{sec:prelims}.  Detailed computations, especially solving $w \ga = w'$ for $\ga$, are carried out in Section~\ref{sec:computations}.  The proofs of the main results are in Section~\ref{sec:main_results}.  The main theorem is illustrated by an example in Section~\ref{sec:example}.  The paper concludes in Section~\ref{sec:other-alphabets} with an extension of the main theorem to alphabets that can be viewed as pseudo-injective submodules of a Frobenius bimodule.

\section{Preliminaries}  \label{sec:prelims}
Let $R$ be a finite ring with $1$.  Denote the group of units of $R$ by $\U$.  Let $\hr$ be the group of complex characters of the additive group of $R$; i.e., $\hr = \{ \pi: R \ra \C^\times: \pi(r_1+r_2) = \pi(r_1)\pi(r_2), r_1, r_2 \in R \}$.  Then $\hr$ is an abelian group under point-wise multiplication of functions, and $\hr$ is a bimodule over $R$ under the following scalar multiplications:  $({}^s \pi)(r) = \pi(rs)$ and $(\pi^s)(r) = \pi(sr)$, for all $r,s \in R$.

Let $A$ be a Frobenius bimodule over $R$;  i.e., $A$ is a bimodule over $R$ such that ${}_RA \cong  {}_R \widehat{R}$ and $A_R \cong \widehat{R}_R$.  In particular, $\size{A} = \size{R}$.  The Frobenius bimodule $A$ admits a character $\chi$ that generates $\ah$ both as a left $R$-module and as a right $R$-module.  This character $\chi$ also has the equivalent properties that $\ker\chi$ contains no nonzero left nor right $R$-submodules.  We will refer to $\chi$ as a \emph{generating character} for $A$.  Details concerning these results may be found in \cite[Section~5.2]{wood:turkey}.

\begin{lemma}  \label{lemma:chi-sum}
If $S \subseteq A$ is a nonzero left or right submodule of $A$, then $\sum_{s \in S} \chi(s) = 0$.
\end{lemma}

\begin{proof}
Because $\ker\chi$ contains no nonzero left or right submodules, there is some $s_0 \in S$ with $\chi(s_0) \neq 1$.  Let $t=s+s_0$ be a change of variables in the given sum:
\[ \sum_{t \in S} \chi(t) = \sum_{s \in S} \chi(s+s_0) = \sum_{s\in S} \chi(s) \chi(s_0) =  \chi(s_0) \sum_{s \in S} \chi(s). \]
Thus, $(1-\chi(s_0)) \sum_{s\in S} \chi(s) =0$, so that $\sum_{s\in S} \chi(s) =0$.
\end{proof}

Let $\RR$ denote the complex semigroup ring of the multiplicative semigroup of $R$.  Explicitly, $\RR$ is the vector space of all functions $\RR=\{ \al: R \ra \C \}$ from $R$ to the complex numbers $\C$, equipped with the multiplicative convolution product $*$:
\[  (\al * \be)( r) = \sum_{st=r} \al(s) \be(t), \]
where the sum is over all pairs $(s,t) \in R \times R$ satisfying $st=r$.

Let $\A = \{ w: A \ra \C\}$ be the vector space of all functions from $A$ to~$\C$.  For $w \in \A$ and $\al \in \RR$, define the (right) multiplicative correlation $w \ca \al \in \A$ by $(w \ca \al)(a) = \sum_{r \in R} w(ra) \al( r)$, for $a \in A$.  In a similar fashion, one can define a left multiplicative correlation; we will not need the left version in this paper.

Define a \emph{Fourier transform} $\widehat{\hphantom{w}}: \RR \ra \A$ by $\widehat{\al} = \chi \ca \al$; i.e., $\widehat{\al}(a) = \sum_{r \in R} \chi(ra) \al( r)$, for $a \in A$.

\begin{lemma} \label{lemma:SW}
The spaces $\RR$ and $\A$ satisfy the following, with $w \in \A$, $\al, \be \in \RR$:
\begin{enumerate}
\item $\RR$ is a finite-dimensional algebra over $\C$; $\dim \RR = |R|$.
\item  $\A$ is a right $\RR$-module under $\ca$: $w \ca (\al * \be) = (w \ca \al) \ca \be$.
\item  $\widehat{\hphantom{w}}: \RR \ra \A$ is an isomorphism of right $\RR$-modules; in particular, $\widehat{\al * \be} = \widehat{\al} \ca \be$.  The inverse transform sends $w \in \A$ to $\widetilde{w}(r )= (1/\size{A}) \sum_{a \in A} w(a) \chi(-ra)$.
\end{enumerate}
\end{lemma}

\begin{proof}
These are routine verifications.  We will show one of them.
\begin{align*}
((w \ca \al) \ca \be)(a) &= \sum_{t \in R} (w \ca \al)(ta) \be(t) \\
&= \sum_{t \in R} ( \sum_{s \in R} w(sta) \al(s) ) \be(t) \\
&= \sum_{r \in R} \sum_{st=r} w(ra) \al(s) \be(t) = (w \ca (\al * \be))(a)
\end{align*}

For the inverse transform claim, see Lemma~\ref{lemma:w-tilde}.
\end{proof}

We define subspaces of $\RR$ and $\A$: let $\RR_0 = \{ \al \in \RR: \sum_{r \in R} \al(r) = 0 \}$ and $\A_0 = \{ w \in \A: w(0)=0 \}$.

\begin{lemma}
\begin{enumerate}
\item For $\al, \be \in \RR$, 
\[  \sum_{r \in R} (\al * \be)( r) = \Big(\sum_{s \in R} \al(s)\Big) \Big(\sum_{t\in R}\be(t)\Big). \]
\item  For $w \in \A$ and $\al \in \RR$, $(w \ca \al)(0) = w(0) \sum_{r \in R} \al(r )$.
\item  For $\al \in \RR$, $\widehat{\al}(0) = \sum_{r \in R} \al( r)$.
\item  $\RR_0$ is a two-sided ideal in $\RR$.
\item  $\A_0$ is a right $\RR$-submodule of $\A$.
\item  $\widehat{\hphantom{w}}: \RR_0 \ra \A_0$ is an isomorphism of right $\RR$-modules.
\end{enumerate}
\end{lemma}

\begin{proof}
Again, these are routine verifications.
\end{proof}

An element $\al \in \RR$ is \emph{bi-invariant} if $\al(u r v) = \al(r)$ for all $r \in R$ and $u,v \in \U$.  Similarly, an element $w \in \A$ is \emph{bi-invariant} if $w(uav)=w(a)$ for all $a \in A$ and $u,v \in \U$.

\begin{lemma}  Let $\al, \be \in \RR$ and $w \in \A$ all be bi-invariant.  Then $\al * \be$, $\charac{\al}$, and $w \ca \al$ are bi-invariant.
\end{lemma}

\begin{proof}
We will verify the claim for $w \ca \al$ and leave the others as exercises for the reader.  
\begin{align*}
(w \ca\al)(uav) &= \sum_{r \in R} w(ruav) \al(r) = \sum_{r \in R} w(rua) \al(r) \\
&= \sum_{s \in R} w(sa) \al(s u^{-1}) = \sum_{s \in R} w(sa) \al(s) = (w \ca\al)(a) ,
\end{align*}
where we made use of the change of variables $s = ru$, $u \in \U$.
\end{proof}

\section{Some Computations}  \label{sec:computations}
Suppose $w \in \A_0$ is bi-invariant.  By making the change of variable $a \leftrightarrow -a$ and using bi-invariance,  the inverse Fourier transform of $w$ is 
\[  \widetilde{w}(r )= \frac{1}{\size{A}} \sum_{a \in A} w(a) \chi(ra) . \]
\begin{lemma}  \label{lemma:w-tilde}
If $w \in \A_0$ is bi-invariant, then $\widetilde{w} \in \RR_0$ is bi-invariant and $\chi \ca \widetilde{w} = w$.
\end{lemma}

\begin{proof}
Let us begin by verifying that $\widetilde{w} \in \RR_0$:
\begin{align*}
\size{A} \sum_{r \in R} \widetilde{w}(r) &= \sum_{r \in R} \sum_{a \in A} w(a) \chi(ra) \\
&= 
\sum_{a \in A} w(a) \sum_{r \in R} \chi(ra) = \size{R} w(0) = 0,
\end{align*} 
where we simplified the summation using Lemma~\ref{lemma:chi-sum} for $S=Ra$:
\[  \sum_{r \in R} \chi(ra) = \begin{cases}
\size{R}, & a = 0, \\
0, & a \neq 0 .
\end{cases} \]

By the change of variable $b = va$, $v \in \U$, one sees that $\widetilde{w}(rv) = \widetilde{w}(r)$.  For $u \in \U$, $\size{A} \widetilde{w}(ur) = \sum_{a \in A} w(a) \chi^u(ra)$.  Because $u$ is a unit, $\chi^u$ is another generating character for $A$.  As such, it must equal ${}^{u'} \chi$ for some unit $u' \in \U$.  Then $\size{A} \widetilde{w}(ur) = \sum_{a \in A} w(a) \chi(rau')$.  Using the change of variable $b = au'$, we conclude that $\widetilde{w}(ur) = \widetilde{w}(r)$.

Finally, using bi-invariance of $w$ so that $w(-b)=w(b)$,
\begin{align*}
\size{A} (\chi \ca \widetilde{w})(a) &= \sum_{r \in R} \chi(ra) \sum_{b \in A} w(b) \chi(-rb) \\
&= \sum_{b \in A} w(b) \sum_{r \in R} \chi(r(a-b)) = \size{A} w(a) ,
\end{align*}
where we used $\size{A} = \size{R}$ and Lemma~\ref{lemma:chi-sum} with $S = R(a-b)$.
\end{proof}

A particular example of a bi-invariant weight on a Frobenius bimodule is the homogeneous weight $w_{\Hom}$.  The existence of the homogeneous weight on any Frobenius bimodule was proved in \cite{greferath-nechaev-wisbauer:modules} and \cite{Honold-Nechaev:weighted-modules}.  We will need the features of the homogeneous weight listed in the next theorem.

\begin{theorem}[\cite{honold:frobenius}] \label{thm:homog}
Suppose $w_{\Hom}$ is the homogeneous weight on a Frobenius bimodule $A$.  Then
\begin{enumerate}
\item  $w_{\Hom}$ is expressible in terms of the generating character $\chi$ of $A$:
\[  w_{\Hom}(a) = 1 - \frac{1}{\size{\U}} \sum_{u \in \U} \chi(ua) ; \]
\item $w_{\Hom}$ is bi-invariant.
\end{enumerate}
\end{theorem}

Define $\varepsilon \in \RR_0$ by 
\[  \varepsilon(r) = \begin{cases}
-1/\size{\U}, & r \in \U, \\
1, & r=0, \\
0, & \text{otherwise}.
\end{cases}  \]

\begin{lemma}  \label{lemma:w-homog}
The Fourier transform of $\varepsilon$ is $w_{\Hom}$, the homogeneous weight of $A$.  That is, $\chi \ca \varepsilon = w_{\Hom}$.
\end{lemma}

\begin{proof}
Calculate and make use of Theorem~\ref{thm:homog}:
\begin{align*}
(\chi \ca \varepsilon)(a) &= \sum_{r \in R} \chi(ra) \varepsilon(r) \\
&= 1 - \frac{1}{\size{\U}} \sum_{u \in \U} \chi(ua) = w_{\Hom}(a) . \qedhere
\end{align*}
\end{proof}

We quote the following lemma:

\begin{lemma}[{\cite[Prop.\ 5.1]{wood:duality}}]  \label{lemma:Bass}
For cyclic right submodules of $A$, $aR = bR$ if and only if $a \U = b \U$.
\end{lemma}

The set $P = \{ aR: a \in A \}$ of all cyclic right submodules of $A$ is a partially ordered set (poset) under inclusion $\subseteq$.  Denote the M\"obius function of $P$ by $\mu$.  Then $\mu$ is characterized by the following properties: 
\begin{itemize}
\item $\mu(aR, cR) = 0$ when $aR \not\subseteq cR$;
\item  $\mu(aR,aR) = 1$;
\item  $\sum_{aR \subseteq bR \subseteq cR} \mu(aR, bR) = 0$, where $bR$ varies for fixed $aR \subsetneq cR$.
\end{itemize}
For further details about posets and their M\"obius functions, see \cite{stanley:enumerative-combinatorics1}.

The next result appears in \cite{honold:frobenius}.
\begin{lemma}  \label{lemma:Mobius}
The M\"obius function of $P$ has $\mu(0,aR) = \sum_{b \in a\U} \chi(b)$.
\end{lemma}

\begin{proof}
Define $g: P \ra \C$ by $g(aR) = \sum_{b \in a \U} \chi(b)$.  We verify that $g$ satisfies the properties of $\mu(0, \cdot)$.

Certainly $g(0) = \chi(0)=1$.  For fixed $cR \neq 0$, use Lemma~\ref{lemma:chi-sum}:
\[  \sum_{bR \subseteq cR} g(bR) = \sum_{b\U \subseteq cR} \sum_{a \in b \U} \chi(a) = \sum_{a \in cR} \chi(a) =0.  \qedhere  \]
\end{proof}

Suppose $w \in \A_0$ is bi-invariant.  Consider $w$ that satisfy:
\begin{equation}  \label{eq:w-condition}
\sum_{aR \subseteq S} w(a) \mu(0,aR) \neq 0, 
\end{equation}
for all nonzero right submodules $S \subseteq A$.

\begin{lemma}
Let $w \in \A_0$ be bi-invariant.  Then \eqref{eq:w-condition} is equivalent to 
\[  \sum_{a \in S} w(a) \chi(a) \neq 0, \quad \text{for all nonzero right submodules $S \subseteq A$} .  \]
\end{lemma}

\begin{proof}
The group of units $\U$ acts on $A$ on the right by scalar multiplication.  Any right submodule $S$ of $A$ is invariant under this action.  Thus $S$ is the disjoint union of right $\U$-orbits.  Break up the sum over $S$ into the sum over these right $\U$-orbits.  Then, by Lemma~\ref{lemma:Mobius},
\begin{align*}
\sum_{a \in S} w(a) \chi(a) &= \sum_{a\U \subseteq S} \sum_{b \in a\U} w(b) \chi(b) 
= \sum_{a\U \subseteq S}  w(a) \sum_{b \in a\U} \chi(b) \\
&= \sum_{aR \subseteq S} w(a) \mu(0,aR) .  \qedhere
\end{align*}
\end{proof}

The next computation is the key technical result in the paper.
\begin{lemma}  \label{lemma:key}
Suppose $w \in \A_0$ is bi-invariant and satisfies \eqref{eq:w-condition} for all nonzero right submodules $S \subseteq A$.  Then there exists $\ga \in \RR_0$ that is bi-invariant and satisfies $\widetilde{w} * \ga = \varepsilon$.
\end{lemma}

\begin{proof}
We define $\ga$ in stages, starting on the group of units $\U$ in $R$, by solving the equation
\begin{equation}  \label{eq:gamma-eqn}
\varepsilon(r) = \sum_{st=r} \widetilde{w}(s) \ga(t) .
\end{equation}

If $r=u \in \U$, then $st=u$ implies that $s$ and $t$ are in $\U$.  Equation~\eqref{eq:gamma-eqn} then becomes $-1/\size{\U} = \sum_{t \in \U} \widetilde{w}(r t^{-1}) \ga(t) = \size{\U} \widetilde{w}(1) \ga(1)$.  We have used the bi-invariance of $\widetilde{w}$ and the desired bi-invariance of~$\ga$. Note that $\widetilde{w}(1) \neq 0$, this being the $S=A$ case of \eqref{eq:w-condition}.  Thus, by defining $\ga(t) = \ga(1) = -1/(\size{\U}^2 \widetilde{w}(1))$, we solve \eqref{eq:gamma-eqn} when $r=u$, and that portion of $\ga$ that has been defined is bi-invariant.

As a recursive step, suppose a solution $\ga$ has been defined for various $r \in R$ such that $\ga$ is bi-invariant; i.e., if $\ga(r)$ is defined, then $\ga(ur)$ and $\ga(ru)$ are also defined for all $u \in \U$, and $\ga(ur)=\ga(ru)=\ga(r)$ for all $u \in \U$.  Now let $r \in R$ be any element, neither a unit nor $0$, such that $Rr$ is maximal among all principal left ideals of $R$ where $\ga$ is not defined on $\U r$.  Consider \eqref{eq:gamma-eqn} for this choice of $r$.

Any solution of the equation $st=r$ has the property that $Rr \subseteq Rt$.  If $Rr \subsetneq Rt$, then the maximality property of $Rr$ implies that $\ga(t)$ is already defined.  Then \eqref{eq:gamma-eqn} becomes
\[  0 = \sum_{\substack{st=r \\ Rr \subsetneq Rt}} \widetilde{w}(s) \ga(t) + \sum_{\substack{st=r \\ Rr = Rt}} \widetilde{w}(s) \ga(t) . \]

We analyze the sum where $Rr=Rt$.  Then $t = ur$ for some $u \in \U$ by Lemma~\ref{lemma:Bass}, so that $st=sur=r$ and $(su-1)r=0$.  That is, $su-1 \in \ann_{\lt}(r)$, where $\ann_{\lt}(r) = \{ q \in R: qr=0 \}$ is the left annihilator of $r$ in $R$; $\ann_{\lt}(r)$ is a left ideal in $R$.  For every $q \in \ann_{\lt}(r)$ and $u \in \U$ we obtain a factorization $st=r$ via $s=(1+q)u^{-1}$ and $t=ur$, and, provided $Rr=Rt$, every factorization is of this form.

Using the bi-invariance of $\widetilde{w}$ and $\ga$, we see that
\begin{align*}
\sum_{\substack{st=r \\ Rr = Rt}} \widetilde{w}(s) \ga(t) &= \sum_{q \in \ann_{\lt}(r), u \in \U} \widetilde{w}((1+q)u^{-1}) \ga(ur) \\
&= \size{\U} \ga(r) \sum_{q \in \ann_{\lt}(r)} \widetilde{w}(1+q) .
\end{align*}
The last summation simplifies in the following way:
\begin{align*}
\size{A} \sum_{q \in \ann_{\lt}(r)} \widetilde{w}(1+q) &= \sum_{q \in \ann_{\lt}(r)} \sum_{a \in A} w(a) \chi((1+q)a) \\
&= \sum_{a \in A} w(a) \chi(a) \sum_{q \in \ann_{\lt}(r)} \chi(qa).
\end{align*}
The summation $\sum_{q \in \ann_{\lt}(r)} \chi(qa)$ is a sum over the left submodule $\ann_{\lt}(r) a \subseteq A$.  Because $\chi$ is a generating character for $A$, this sum vanishes (Lemma~\ref{lemma:chi-sum}) unless the submodule $\ann_{\lt}(r) a = 0$, in which case the sum equals $\size{\ann_{\lt}(r)}$.

For $r \in R$, define $S_r = \{ a \in A: \ann_{\lt}(r) a = 0 \}$; $S_r$ is a right submodule of $A$.  Continuing the simplification from above, we see that
\[
\size{A} \sum_{q \in \ann_{\lt}(r)} \widetilde{w}(1+q) = \size{\ann_{\lt}(r)} \sum_{a \in S_r} w(a) \chi(a) .
\]
By the $S=S_r$ case of \eqref{eq:w-condition}, we see that $\sum_{q \in \ann_{\lt}(r)} \widetilde{w}(1+q) \neq 0$.  This allows us to solve for $\ga(r)$ in terms of previously defined values of $\ga$:
\[  \ga(r) = -\bigg(\sum_{\substack{st=r \\ Rr \subsetneq Rt}} \widetilde{w}(s) \ga(t)\bigg) / \bigg( \size{\U} \sum_{q \in \ann_{\lt}(r)} \widetilde{w}(1+q) \bigg).  \]
This formula gives the same value for $\ga(ru)$, $u \in \U$, and the entire derivation is the same if $r$ is replaced by $ur$.  Thus $\ga$ is bi-invariant.

Proceeding recursively, we eventually arrive at the case $r=0$.  Then \eqref{eq:gamma-eqn} becomes
\[  1= \sum_{\substack{st=0 \\ t \neq 0}} \widetilde{w}(s) \ga(t) + \sum_{s \in R} \widetilde{w}(s) \ga(0) . \]
Because $\widetilde{w} \in \RR_0$, the coefficient $\sum_{s \in R} \widetilde{w}(s)$ of $\ga(0)$ vanishes.  Thus, the equation puts no restriction on $\ga(0)$.  By setting $\ga(0) = -\sum_{r \neq 0} \ga(r)$, we have $\ga \in \RR_0$.

Why does $(\widetilde{w} * \ga)(0) = 1$?  We know $\widetilde{w} \in \RR_0$, and we have already shown that $\ga \in \RR_0$ and $(\widetilde{w} * \ga)(r) = \varepsilon(r)$ for all $r \neq 0$.  Since $\widetilde{w} * \ga$ and $\varepsilon$ belong to $\RR_0$, we conclude that their values at $r=0$ also agree.
\end{proof}

\begin{corollary}  \label{cor:w-gamma}
Suppose $w \in \A_0$ is bi-invariant and satisfies \eqref{eq:w-condition} for all nonzero right submodules $S \subseteq A$.  Then $w \ca \ga = w_{\Hom}$.
\end{corollary}

\begin{proof}
Use Lemmas~\ref{lemma:SW}, \ref{lemma:w-tilde}, \ref{lemma:w-homog}, and~\ref{lemma:key}.
Then, $w \ca \ga = (\chi \ca \widetilde{w}) \ca \ga = \chi \ca (\widetilde{w} * \ga) = \chi \ca \varepsilon = w_{\Hom}$.
\end{proof}

\section{The Main Theorem}  \label{sec:main_results}

As in earlier sections, $R$ is a finite ring with $1$ and $A$ is a Frobenius bimodule over $R$.
An \emph{$R$-linear code} over $A$ of length $n$ is a left $R$-submodule $C \subseteq A^n$.  A function $w \in \A_0$ defines a \emph{weight} on $A^n$ (and, by restriction, on $C$) by $w(x) = \sum_{i=1}^n w(x_i)$, where $x = (x_1, x_2, \ldots, x_n) \in A^n$.  An injective $R$-linear homomorphism $f: C \ra A^n$ is a \emph{$w$-isometry} if $w(f(x)) = w(x)$ for all $x \in C$.

\begin{proposition}  \label{prop:isometry}
Suppose $f: C \ra A^n$ is a $w$-isometry.  Then $f$ is a $w \ca \al$-isometry for any $\al \in \RR$.
\end{proposition}

\begin{proof}
For $x \in C$, 
\begin{align*}
(w \ca \al)(f(x)) &= \sum_{r \in R} w(rf(x)) \al(r) = \sum_{r \in R} w(f(rx)) \al(r) \\
&= \sum_{r \in R} w(rx) \al(r) = (w \ca \al)(x).  \qedhere
\end{align*}
\end{proof}

For units $u_1, u_2, \ldots, u_n \in \U$ and a permutation $\sigma$ of $\{ 1, 2, \ldots, n \}$, the left $R$-module homomorphism $T: A^n \ra A^n$ defined by 
\[  T(x_1, x_2, \ldots, x_n) = (x_{\sigma(1)} u_1, x_{\sigma(2)} u_2, \ldots, x_{\sigma(n)} u_n) , \]
$(x_1, x_2, \ldots, x_n) \in A^n$, is called a \emph{monomial transformation} of $A^n$.  If $w$ is bi-invariant, then any monomial transformation is a $w$-isometry.

\begin{defn}
A bi-invariant weight $w$ on a Frobenius bimodule $A$ has the \emph{extension property} if the following holds:  for any linear code $C \subseteq A^n$ and any $w$-isometry $f: C \ra A^n$, the isometry $f$ extends to a monomial transformation of $A^n$.
\end{defn}

An example of a bi-invariant weight that has the extension property is the homogeneous weight $w_{\Hom}$:
\begin{theorem}[{\cite[Theorem~4.10]{greferath-nechaev-wisbauer:modules}}] \label{thm:GNW}
Any Frobenius bimodule $A$ has the extension property with respect to $w_{\Hom}$.
\end{theorem}

\begin{theorem}  \label{thm:main}
Suppose $A$ is a Frobenius bimodule over $R$ and $w$ is a bi-invariant weight on $A$.  If $w$ satisfies \eqref{eq:w-condition} for all nonzero right submodules $S \subseteq A$, then $w$ has the extension property.
\end{theorem}

\begin{proof}
Suppose $C \subseteq A^n$ is a linear code and $f: C \ra A^n$ is a $w$-isometry.  Then $f$ is an isometry with respect to $w \ca \ga = w_{\Hom}$, Corollary~\ref{cor:w-gamma} and Proposition~\ref{prop:isometry}.  By Theorem~\ref{thm:GNW}, $w_{\Hom}$ has the extension property, so $f$ extends to a monomial transformation, as desired.
\end{proof}

\begin{corollary}  \label{cor:Frob-EP}
Suppose $R$ is a finite Frobenius ring and $w$ is a bi-invariant weight on $R$.  If $w$ satisfies \eqref{eq:w-condition} for all nonzero right ideals $S \subseteq R$, then $w$ has the extension property.
\end{corollary}

\begin{rem}
Corollary~\ref{cor:Frob-EP} generalizes to all finite Frobenius rings the condition proved for finite principal ideal rings in \cite[Theorem~4.4]{MR3207471}.
\end{rem}

\section{Example}  \label{sec:example}

We illustrate Theorem~\ref{thm:main} with an example.

\begin{example}
Let $R = U_2(\F_2)$ be the ring of $2 \times 2$ upper-triangular matrices over $\F_2$.  The ring $R$ is an algebra over $\F_2$ of dimension $3$, so that $\size{R}=8$.  When $r \in R$ is the matrix 
\[  r = \begin{bmatrix}
a & b \\
0 & c
\end{bmatrix}, \]
we will denote the element $r$ with the notation $r = [a,b,c]$.   It is known that the ring $R$ is not Frobenius. 

The Frobenius bimodule $A=\hr$ is a vector space over $\F_2$ of dimension $3$.  For $\pi \in A$, we will use the notation $\pi = (x,y,z)$, where the evaluation $\pi(r)$, $r = [a,b,c] \in R$, is 
\[  \pi(r) = (-1)^{ax+by+cz} . \]
The reader is invited to verify the following formulas describing the scalar multiplications on the bimodule $A$ over $R$.  As above, $r = [a,b,c]$ and $\pi = (x,y,z)$:
\begin{align*}
{}^r \pi &= (ax+by, cy, cz), \\
\pi^r &= (ax, ay, by +cz) .
\end{align*}

In order to understand \eqref{eq:w-condition}, we need to know the right submodules of $A$.  For $\pi = (x,y,z) \in A$, we will denote the right submodule generated by $\pi$ as $(x,y,z) R$.  By straight-forward computations using the formulas above, we have the following list of right submodules of $A$.  The module $A$ is not a cyclic module, while all its proper submodules are cyclic.
\begin{description}
\item[$\dim 0$]  $0=(0,0,0)R$;
\item[$\dim 1$]  $(0,0,1)R$, $(1,0,0)R$;
\item[$\dim 2$] $(0,1,0)R = (0,1,1)R$, $(1,0,1)R$, $(1,1,0)R = (1,1,1)R$;
\item[$\dim 3$] $A$ itself.
\end{description}
Also, $(1,0,0)R \subset (1,0,1)R$, and $(0,0,1)R$ is contained in every submodule of dimension $2$.  

By using the containments just described, we calculate that the values of the M\"obius function $\mu(0,aR)$ are as follows:
\[  \begin{array}{c|cccccc}
a & 0 & (0,0,1) & (1,0,0) & (0,1,0) & (1,0,1) & (1,1,0) \\ \hline
\mu(0,aR) & 1 & -1 & -1 & 0 & 1 & 0 
\end{array}. \]

Theorem~\ref{thm:main} applies to bi-invariant weights.  The description above of the right cyclic submodules implies that $(0,1,0)\U = (0,1,1)\U$ and $(1,1,0)\U = (1,1,1)\U$, with all other right $\U$-orbits being distinct.  Similar computations of left $\U$-orbits show that all the left $\U$-orbits are distinct except for $\U(0,1,0)= \U(1,1,0)$ and $\U(0,1,1) = \U(1,1,1)$.  Thus, a weight $w$ on $A$ is bi-invariant if and only if 
\[  w(0,1,0)=w(0,1,1) = w(1,1,0) = w(1,1,1).  \]

Theorem~\ref{thm:main} says that a bi-invariant weight $w$ has the extension property when the following expressions are nonzero:
\[  \begin{array}{l|l}
\text{right module $S$} & \text{expression in \eqref{eq:w-condition}} \\ \hline
(0,0,1)R & -w(0,0,1) \\
(1,0,0)R & -w(1,0,0) \\
(0,1,0)R & -w(0,0,1) \\
(1,0,1)R & -w(0,0,1)-w(1,0,0)+w(1,0,1) \\
(1,1,0)R & -w(0,0,1) \\
A & -w(0,0,1)-w(1,0,0)+w(1,0,1)
\end{array} \]
In particular, the value of $w(0,1,0)$ is irrelevant.
\end{example}

\section{Other Alphabets}  \label{sec:other-alphabets}
In this section we consider the extension problem for certain other alphabets.

Suppose $R$ is a finite ring with $1$ and $A$ is a finite unital left $R$-module.  Let $w$ be a weight on $A$, i.e., a function $w: A \ra \C$ with $w(0)=0$.  Then $w$ extends to a weight on $A^n$ by $w(x) = \sum_{i=1}^n w(x_i)$, where $x = (x_1, x_2, \ldots, x_n) \in A^n$.  The group of all invertible homomorphisms $\phi: A \ra A$ is denoted $\GL_R(A)$, and $\phi$ will be written with inputs on the left, so that $(ra)\phi = r(a\phi)$ for all $r \in R$ and $a \in A$.  The \emph{right symmetry group} of $w$ on $A$ is 
\[  \grt = \{ \phi \in \GL_R(A): w(a \phi)  = w(a), a\in A \} . \]
A \emph{$\grt$-monomial transformation} of $A^n$ is any invertible homomorphism $T: A^n \ra A^n$ of the form
\[   (x_1, x_2, \ldots, x_n) T = (x_{\sigma(1)} \phi_1, x_{\sigma(2)} \phi_2, \ldots, x_{\sigma(n)} \phi_n) , \]
for $(x_1, x_2, \ldots, x_n) \in A^n$, where $\sigma$ is a permutation of $\{1, 2, \ldots, n\}$ and $\phi_i \in \grt$.  Observe that a $\grt$-monomial transformation preserves $w$: $w(xT) = w(x)$ for all $x \in A^n$.

A weight $w$ on $A$ has the \emph{extension property} if, for any left $R$-linear code (submodule) $C \subseteq A^n$ and any injective homomorphism $f: C \ra A^n$ of left $R$-modules that preserves $w$, $w(c) = w(f(c))$, $c \in C$, it follows that $f$ extends to a $\grt$-monomial transformation of $A^n$.

We will prove that the extension property holds for certain alphabets.  Let $A$ be a finite unital left $R$-module that is pseudo-injective and such that its socle $\soc(A)$ is a cyclic left $R$-module.  Recall that a left $R$-module is \emph{pseudo-injective} if for any left $R$-submodule $B \subseteq A$ and any injective homomorphism $\phi: B \ra A$ of left $R$-modules, $\phi$ extends to a homomorphism $\tilde{\phi}: A \ra A$ of left $R$-modules.  It follows from \cite[Proposition~3.2]{dinh-lopez:local} that there exists an extension $\tilde{\phi}$ that is injective, hence invertible.  The \emph{socle} $\soc(A)$ of $A$ is the left submodule generated by all simple left $R$-submodules of $A$.  The hypothesis that $\soc(A)$ is cyclic implies that $A$ injects into $\hr$ as a left $R$-module \cite[Proposition~5.3]{wood:turkey}.

\begin{theorem}  \label{thm:OtherAlphabets}
Suppose $R$ is a finite ring with $1$ and $A$ is a finite unital left $R$-module that is pseudo-injective and has $\soc(A)$ cyclic.  Fix an injection of $A$ into $\hr$, and view $A$ as a left submodule of $\hr$.  Let $w$ be a weight on $A$, and assume that $w$ is the restriction to $A$ of a bi-invariant weight on $\hr$ that satisfies \eqref{eq:w-condition}.  Then the weight $w$ on the alphabet $A$ satisfies the extension property.
\end{theorem}

Before proving Theorem~\ref{thm:OtherAlphabets} we state the following lemma.

\begin{lemma} \label{lem:extend}
Suppose $R$ is a finite ring with $1$, $I$ is a finite injective left module over $R$, and $A \subseteq I$ is a left $R$-submodule.  Then every element of $\GL_R(A)$ extends to an element of $\GL_A(I)$.
\end{lemma}

\begin{proof}
This is essentially \cite[Proposition~3.2]{dinh-lopez:local}, discussed above.  Any injective module is, a fortiori, pseudo-injective.  Injections $A \ra A \subseteq I$ then extend to injections $I \ra I$.
\end{proof}

\begin{proof}[Proof of Theorem~\ref{thm:OtherAlphabets}]
Suppose $C \subseteq A^n$ is a left $R$-linear code and $f: C \ra A^n$ is an injective homomorphism that preserves $w$.  Because $A \subseteq \hr$, we can view $C \subseteq \hr^n$ as an $R$-linear code over the alphabet $\hr$ and $f: C \ra \hr^n$;  $f$ still preserves $w$, now viewed as a bi-invariant weight on $\hr$ that satisfies \eqref{eq:w-condition}.  By Theorem~\ref{thm:main}, $f$ extends to a monomial transformation $T: \hr^n \ra \hr^n$.  

Write the components of $f: C \ra A^n$ as $f=(f_1, f_2 \ldots, f_n)$, $f_i : C \ra A$.  The fact that $f$ extends to the monomial transformation $T$ means there exist a permutation $\sigma$ of $\{1, 2, \ldots, n\}$ and units $u_i \in \U(R)$ such that $f_i(x) = x_{\sigma(i)} u_i$, for $x=(x_1, x_2, \ldots, x_n) \in C$.  Because $f: C \ra A^n$, this means that $x_{\sigma(i)} u_i \in A$ for $x \in C$.

Let $C_i = \{ x_i: x \in C \}$ be the projection of $C \subseteq A^n$ on the $i$th coordinate; $C_i \subseteq A$ is an $R$-submodule.  The extendability of $f$ implies that right multiplication by $u_i$ injects $C_{\sigma(i)} \ra A$.  Because $A$ is pseudo-injective, this injection extends to an injection $A \ra A$, i.e., to an element $\phi_i \in \GL_R(A)$.  Thus, $f_i(x) = x_{\sigma(i)} \phi_i$, $x \in C$.

The last thing we need to prove is that $\phi_i \in \grt$, i.e., that $\phi_i$ preserves $w$ on $A$.  This follows from Lemma~\ref{lem:extend}.  Indeed, $\hr$ is always an injective module, so $\phi_i \in \GL_R(A)$ extends to an element of $\GL_R(\hr)$.  But $\GL_R(\hr)= \U(R)$, and $w$ is bi-invariant as a weight on $\hr$, so every element of $\GL_R(\hr)$ preserves $w$.  By restriction, every element of $\GL_R(A)$ preserves $w$ on $A$.
\end{proof}

\bibliographystyle{amsplain} 
\bibliography{coding}

\def\cprime{$'$} \def\cprime{$'$} \def\cprime{$'$} \def\cprime{$'$}
  \def\cprime{$'$}
\providecommand{\bysame}{\leavevmode\hbox to3em{\hrulefill}\thinspace}
\providecommand{\MR}{\relax\ifhmode\unskip\space\fi MR }
\providecommand{\MRhref}[2]{%
  \href{http://www.ams.org/mathscinet-getitem?mr=#1}{#2}
}
\providecommand{\href}[2]{#2}
\begin{thebibliography}{10}

\bibitem{dinh-lopez:local}
H.~Q. Dinh and S.~R. L{\'o}pez-Permouth, \emph{On the equivalence of codes over
  finite rings}, Appl.\ Algebra Engrg.\ Comm.\ Comput. \textbf{15} (2004),
  no.~1, 37--50. \MR{2142429 (2006d:94097)}

\bibitem{GH-OnWeights}
M.~Greferath and T.~Honold, \emph{On weights allowing for {MacWilliams}
  equivalence theorem}, Proceedings of the Fourth International Workshop in
  Optimal Codes and Related Topics (Pamporovo, Bulgaria), 2005, pp.~182--192.

\bibitem{GH-MonomialExtensions}
\bysame, \emph{Monomial extensions of isometries of linear codes {II}:
  {Invariant} weight functions on {$Z_m$}}, Proceedings of the Tenth
  International Workshop in Algebraic and Combinatorial Coding Theory (ACCT-10)
  (Zvenigorod, Russia), 2006, pp.~106--111.

\bibitem{MR3207471}
M.~Greferath, Th. Honold, C.~Mc~Fadden, J.~A. Wood, and J.~Zumbr{\"a}gel,
  \emph{Mac{W}illiams' extension theorem for bi-invariant weights over finite
  principal ideal rings}, J. Combin. Theory Ser. A \textbf{125} (2014),
  177--193. \MR{3207471}

\bibitem{MR3016561}
M.~Greferath, C.~Mc~Fadden, and J.~Zumbr{\"a}gel, \emph{Characteristics of
  invariant weights related to code equivalence over rings}, Des. Codes
  Cryptogr. \textbf{66} (2013), no.~1-3, 145--156. \MR{3016561}

\bibitem{greferath-nechaev-wisbauer:modules}
M.~Greferath, A.~Nechaev, and R.~Wisbauer, \emph{Finite quasi-{F}robenius
  modules and linear codes}, J. Algebra Appl. \textbf{3} (2004), no.~3,
  247--272. \MR{2096449 (2005g:94099)}

\bibitem{greferath-schmidt:combinatorics}
M.~Greferath and S.~E. Schmidt, \emph{Finite-ring combinatorics and
  {M}ac{W}illiams's equivalence theorem}, J. Combin.\ Theory Ser.\ A
  \textbf{92} (2000), no.~1, 17--28. \MR{1783936 (2001j:94045)}

\bibitem{honold:frobenius}
T.~Honold, \emph{Characterization of finite {F}robenius rings}, Arch.\ Math.\
  (Basel) \textbf{76} (2001), no.~6, 406--415. \MR{1831096 (2002b:16033)}

\bibitem{Honold-Nechaev:weighted-modules}
T.~Honold and A.~A. Nechaev, \emph{Weighted modules and representations of
  codes}, Problems Inform.\ Transmission \textbf{35} (1999), no.~3, 205--223.
  \MR{1730800 (2001f:94016)}

\bibitem{macwilliams:theorem}
F.~J. MacWilliams, \emph{Error-correcting codes for multiple-level
  transmission}, Bell System Tech.\ J. \textbf{40} (1961), 281--308.
  \MR{0141541 (25 \#4945)}

\bibitem{macwilliams:thesis}
\bysame, \emph{Combinatorial problems of elementary abelian groups}, Ph.D.
  thesis, Radcliffe College, Cambridge, Mass., 1962.

\bibitem{stanley:enumerative-combinatorics1}
R.~P. Stanley, \emph{Enumerative combinatorics}, The Wadsworth \& Brooks/Cole
  Mathematics Series, vol.~I, Wadsworth \& Brooks/Cole Advanced Books \&
  Software, Monterey, CA, 1986. \MR{847717 (87j:05003)}

\bibitem{wood:duality}
J.~A. Wood, \emph{Duality for modules over finite rings and applications to
  coding theory}, Amer.\ J. Math. \textbf{121} (1999), no.~3, 555--575.
  \MR{1738408 (2001d:94033)}

\bibitem{wood:waterloo}
\bysame, \emph{Weight functions and the extension theorem for linear codes over
  finite rings}, Finite fields: theory, applications, and algorithms (Waterloo,
  ON, 1997) (R.~C. Mullin and G.~L. Mullen, eds.), Contemp.\ Math., vol. 225,
  Amer. Math. Soc., Providence, RI, 1999, pp.~231--243. \MR{1650644
  (2000b:94024)}

\bibitem{wood:turkey}
\bysame, \emph{Foundations of linear codes defined over finite modules: the
  extension theorem and the {M}ac{W}illiams identities}, Codes over rings
  (Ankara, 2008), Ser.\ Coding Theory Cryptol., vol.~6, World Sci.\ Publ.,
  Hackensack, NJ, 2009, pp.~124--190. \MR{2850303}

\end{thebibliography}

\end{document}